\documentclass[a4paper,11pt]{article}
\usepackage[utf8]{inputenc}
\usepackage{amsmath, amsfonts, amsthm,amssymb,  bbm}
\usepackage{math}
\usepackage{hyperref} 
\usepackage[usenames, dvipsnames]{color}
\newtheorem{lemma}{Lemma}[section]
\newtheorem{theorem}[lemma]{Theorem}

\newtheorem{corollary}[lemma]{Corollary}
\newtheorem{remark}[lemma]{Remark}

\newtheorem{definition}[lemma]{Definition}

 \usepackage{mathrsfs}

\def\beq{\begin{equation}}   \def\eeq{\end{equation}}
\def\bea{\begin{eqnarray}}  \def\eea{\end{eqnarray}}
\newcommand{\lan}{\langle}
\newcommand{\ran}{\rangle}

\newcommand{\di}{{\rm d}}

\newcommand{\n}{{(N)}}

\renewcommand{\e}{{\rm e}}
\newcommand{\ao}{{\rm an}}
\newcommand{\hos}{{\rm ho}}

\renewcommand{\r}{\rho}
\def\norma#1{\left\|#1\right\|}

\def\uno{{\bf 1}}

\numberwithin{equation}{section}

\title{Growth of Sobolev norms for abstract linear Schr\"odinger equations}

\author{
D. Bambusi\footnote{Dipartimento di Matematica Federigo Enriques, Universit\`a degli Studi di Milano, Via Saldini 50, I-20133
Milano, Italy \newline
 \textit{Email: } \texttt{dario.bambusi@unimi.it}},
B. Gr\'ebert\footnote{Laboratoire de Math\'ematiques Jean Leray, Universit\'e de Nantes, 2 rue de la Houssini\`ere
BP 92208, 44322 Nantes Cedex 3, France \newline
 \textit{Email: } \texttt{benoit.grebert@univ-nantes.fr}} ,
A. Maspero\footnote{ International School for Advanced Studies (SISSA), Via Bonomea 265, 34136, Trieste, Italy \newline
 \textit{Email: } \texttt{alberto.maspero@sissa.it}}, 
  D. Robert\footnote{Laboratoire de Math\'ematiques Jean Leray, Universit\'e de Nantes, 2 rue de la Houssini\`ere
BP 92208, 44322 Nantes Cedex 3, France \newline
 \textit{Email: } \texttt{didier.robert@univ-nantes.fr}}
}

\date{\today}

\begin{document}

	\maketitle

\begin{abstract}
We prove an abstract theorem giving a $\langle t\rangle^\epsilon$
bound ($\forall \epsilon>0$)  on the growth of the Sobolev norms in linear
Schr\"odinger equations of the form $\im \dot \psi = H_0 \psi + V(t)
\psi $ when the time $t \to \infty$. The abstract theorem is applied to several cases, including
the cases where (i) $H_0$ is the Laplace operator on a Zoll manifold
and $V(t)$ a pseudodifferential operator of order smaller than 2;
(ii) $H_0$ is the (resonant or nonresonant) Harmonic oscillator in
$\R^d$ and $V(t)$ a pseudodifferential operator of order smaller
than $H_0$ depending in a quasiperiodic way on time. The proof is obtained by first conjugating the system to
some normal form in which the perturbation is a smoothing operator
and then applying the results of \cite{MaRo}.
\end{abstract}


\section{Introduction} 

In this paper we study growth of Sobolev norms for solutions
of the abstract linear Schr\"odinger equation 
\begin{equation}
\label{LS}
\im \partial_t  \psi = H_0 \psi + V(t) \psi \ , 
\end{equation}
in a scale of Hilbert spaces $\cH^r$; here $V(t)$ is a time dependent
operator and $H_0$ a time independent linear operator. We will prove
some abstract results ensuring that for any $r\geq0$ and any
$\epsilon>0$, the $\cH^r$ norm of the solution grows in time at most
as $\langle t\rangle^\epsilon$ as $t \to \infty$, where $\la t \ra := \sqrt{1+t^2}$. The main novelty of
our results is that they allow (1) to weaken the standard gap
assumptions on the spectrum of $H_0$, in particular to deal with some
cases where the gaps are dense in $\R$, and (2) to deal with
perturbations which are of any order strictly smaller than that of
$H_0$ (see below for a precise definition).

The main applications are to the case where
\begin{itemize}
\item[(i)] $H_0$ is either  the Laplace operator on a Zoll manifold
  (e.g. the spheres) or an anharmonic oscillator in $\R$, while $V$ is
  an operator depending arbitrarily on time and having order strictly
  smaller than $H_0$;
\item[(ii)] $H_0$ is the (possibly nonresonant) multidimensional
  Harmonic oscillator and $V(t)$ is an operator which {\it depends on
    time in a quasiperiodic way} and has order strictly smaller than
  $H_0$.
\end{itemize}
Further applications will be presented in the paper.

We emphasize in particular the results (ii) which, as far as we know
are the first controlling growth of Sobolev norms in higher
dimensional systems without any gap condition.

The proof is based on the combination of the ideas of
\cite{Bam16I, Bam17b,BGMR1} (which in turn are a developments of the ideas
of \cite{BBM14}, see also \cite{PT01,IPT05}) and the results of
\cite{MaRo}; precisely, for any positive $N$, we construct a (finite)
sequence of unitary time dependent transformations conjugating $H_0+V(t)$ to a
Hamiltonian of the form
\begin{equation}
  \label{Z.1}
H_0+Z^{(N)}(t)+V^{(N)}(t)\ ,
\end{equation}
where $[H_0;Z^{(N)}]=0$ and $V^{(N)}$ is a smoothing operator of order
$N$, namely an operator belonging to $\cL(\cH^s;\cH^{s+N})$ for any
$s$ (linear bounded operators from $\cH^s$ to $\cH^{s+N}$). Then we
apply Theorem 1.5 of \cite{MaRo} to \eqref{Z.1} getting the $\langle
t\rangle^\epsilon$ bound on the growth of Sobolev norms. 

We think that a further point of interest of our paper is that the
conjugation to a system of the form \eqref{Z.1} is here developed in
an abstract context, instead then in the framework of classes of
pseudodifferential operators adapted to the situation under study;
this is the main reason why we get an abstract theory directly
applicable to many different contexts.

The main point is that we introduce an abstract graded algebra
  of operators whose properties mimic the properties of
  pseudodifferential operators. The use of this framework is made
  possible by the technique we develop to solve the homological
  equations met in the construction of the conjugation of $H$ to
  \eqref{Z.1}. Indeed, we recall that in previous papers the smoothing
  theorem, namely the result conjugating the original system to
  \eqref{Z.1} was obtained by quantizing the procedure of classical
  normal form. Here instead, we work directly at the quantum level, in
  particular solving at this level the two homological equations that
  we find (see eqs. \eqref{hom1} and \eqref{hom4} below).

It is worth to add a few words on the way we solve the homological
equations. 
When dealing with systems related to the applications (i),
we assume that $H_0=f(K_0)$ where $f$ is a superlinear
function and $K_0$ is an operator s.t.
\begin{equation}
  \label{sp.i}
{\rm spec}(K_0)\subset \N+\lambda\ ,\quad \lambda > 0\ .
\end{equation}
In this case we solve the
homological equation essentially by averaging over the flow $e^{-\im t
K_0}$ of $K_0$.
 In
turn this is made possible by the use of a commutator expansion lemma
proved in \cite{dege}. 
When dealing with the $d$ dimensional harmonic
oscillators instead, we take
$$
H_0=\sum_{j=0}^{d}\nu_jK_j\ ,
$$ with $K_j$ commuting linear operators, each one fulfilling the
property \eqref{sp.i}
(think of $K_j=-\partial_{x_j}^2+x_j^2$) and $\nu_j>0$; then we
consider operators of the form 
$$
{\rm e}^{\im\tau\cdot K} \, A \, {\rm e}^{-\im\tau\cdot K}
$$
(where of course $\tau\cdot K:=\tau_1 K_1+...+\tau_d K_d$), remark
that they are quasiperiodic in the ``angles'' $\tau$, and use a Fourier
expansion in $\tau$ in order to solve the homological equation. 

\vskip10pt
The study of growth of Sobolev norms and the related results on the
nature of the spectrum of the Floquet operator has a long history: we
recall the results by \cite{How89, How92, joy} showing that the Floquet
spectrum of systems with growing gaps and bounded perturbations is
pure point, a result which implies boundedness of the expectation
value of the energy.
The first $\la t \ra^\epsilon$ estimates on the 
expectation value of the energy for system of the form \eqref{LS} was
obtained by Nenciu in  \cite{nen} for the case of increasing gaps and
bounded perturbations (see also \cite{barjoy, joy2} for similar results), and by 
 Duclos, Lev  and S{\v t}ov\'\i{\v c}ek \cite{duclos} in case of  shrinking
gaps. In the case of increasing gaps, such results were
improved recently by two of us (see \cite{MaRo}) who obtained the
$\langle t\rangle^\epsilon$ growth of Sobolev norms also in the case
of unbounded perturbations depending arbitrarily on time, for example
in the case where $H_0=-\partial_{x}^2+x^{2k}$, the result of
\cite{MaRo} allows to deal with perturbations growing at infinity as
$|x|^m$ with $m<k-1$. In the present paper we get the result for any
$m<2k$. The result of \cite{MaRo} also applies to perturbations of the
free Schr\"odiger equation on Zoll manifolds with perturbations of
order strictly smaller than 1. Here we deal with perturbations of
order strictly smaller than 2. 
A study of perturbations of maximal order has been done independently
by Montalto \cite{Mon17} who got a control of the growth of Sobolev
norms for the Schr\"odinger equation on $\T$ with
$H=a(t,x)\left|-\partial_{xx}\right|^M+V(t)$ with $M>1/2$, $a$ a smooth
positive function and $V$ a pseudodifferential operator of order
smaller than $M$.

Finally we  recall that in \cite{MaRo}
logarithmic estimates for the growth of Sobolev norms were also
obtained in the case of  perturbations depending analytically on time. Here we do not attack
the problem of getting logarithmic estimates, but we think that our
technique would also allow to get such estimates. 

A remarkable further result was obtained by Bourgain
\cite{Bourgain1999} who obtained a logarithmic bound on the growth of
Sobolev norms for the Schr\"odinger equation on $\T^d$ ($d =1,2$) in
the case of an analytic perturbation depending quasiperiodically on
time. Such a result is based on the use of a Lemma on the clustering
of resonant sites (in a suitable {space time} lattice) which
does not seem to extend to different geometries. The result of
Bourgain was extended by Wang \cite{wang08} to deal with Schr\"odinger
equations on $\T$ perturbed by a potential analytic in time (but
otherwise depending arbitrary on time) and greatly simplified by
Delort \cite{del10} who used it in an abstract framework which allows
to deal with the case of $\T^d$ (any $d \geq 1$) and also with the
case of Zoll manifolds, obtaining a growth bounded by $\la t
\ra^\epsilon$ (see also \cite{fang} for analytic potentials on
$\T^d$). We also mention the reducibility result by \cite{EK09}
dealing with small quasiperiodic perturbations of the free
Schr\"odinger equation on $\T^d$; for such a system, the authors prove
that growth of Sobolev norms cannot happen, provided the frequency of
the quasiperiodic solution {is chosen in a nonresonant set}. At
present our method does not allow to deal with the Schr\"odinger
equation on $\T^d$ for $d\geq 2$.

Concerning Harmonic oscillators in $\R^d$ with $d>1$, a couple of
reducibility results are known, namely \cite{GP} in which the authors
study small bounded perturbations of the {\it completely resonant}
Harmonic oscillator, and \cite{BGMR1} in which we studied small {\it
  polynomial} perturbations of the resonant or nonresonant Harmonic
oscillator.

As far as we know no results are known on growth of Sobolev norms for
perturbations of the harmonic oscillator:
\begin{equation}
  \label{nr}
H_0:=-\Delta+\sum_{j=1}^d\nu_j^2x_j^2\ ,
\end{equation}
with nonresonant frequencies $\nu_j$. This is due to the fact that the
differences between two of its eigenvalues
$\left\{\lambda_a\right\}_{a\in\N^d}$, namely
$$
\lambda_a-\lambda_b=\nu\cdot(a-b)
$$ are dense on the real axis and this prevents the use of {\it any}
previous technique. As anticipated above here we obtain the $\langle
t\rangle^{\epsilon}$ growth for the case of perturbation of order
strictly smaller than the order of the Harmonic oscillator.

\vspace{2em}
\noindent{\bf Acknowledgments.} During the preparation of this work, we were supported 
 by ANR -15-CE40-0001-02  ``BEKAM'' of the Agence Nationale de la Recherche.
A. Maspero is also  partially supported by PRIN 2015 ``Variational methods, with applications to problems in mathematical physics and geometry".

\section{Main results}
\subsection{An abstract graded algebra}\label{sec:alg}
We start with a Hilbert space $\cH$ and a reference operator $K_0$,
which we assume to be selfadjoint and positive, namely such that
$$
\langle \psi; K_0\psi\rangle\geq c_K \norma{\psi}^2\ ,\quad \forall
\psi\in D(K_0^{1/2})\ ,\quad c_K>0\ , 
$$ and define as usual a scale of Hilbert spaces by $\cH^r=D(K_0^r)$
(the domain of the operator $K_0^r$) if $r\geq 0$, and
$\cH^{r}=(\cH^{-r})^\prime$ (the dual space) if $ r<0$.  Finally we
denote by $\cH^{-\infty} = \bigcup_{r\in\R}\cH^r$ and $\cH^{+\infty} =
\bigcap_{r\in\R}\cH^r$.  We endow $\cH^r$ with the natural norm
$\norm{\psi}_r:= \norm{(K_0)^r \psi}_{0}$, where $\norm{\cdot}_0$ is
the norm of $\cH^0 \equiv \cH$. Notice that for any $m\in\R$,
$\cH^{+\infty}$ is a dense linear subspace of $\cH^m$ (this is a
consequence of the spectral decomposition of $K_0$).
      
We introduce now a graded algebra $\cA$ of operators which mimic some
fundamental properties of different classes of pseudo-differential
operators.  For $m\in\R$ let $\cA_m$ be a linear subspace of $
\bigcap_{s\in\R}\cL(\cH^s,\cH^{s-m})$ and define
$\cA:=\bigcup_{m\in\R}\cA_m$.  We notice that the space
$\bigcap_{s\in\R}\cL(\cH^s,\cH^{s-m})$ is a Fr\'echet space equipped
with the semi-norms: $\Vert A\Vert_{m,s} := \Vert
A\Vert_{\cL(\cH^s,\cH^{s-m})}$.

One of our aims is to control the smoothing properties of the operators
in the scale $\{\cH^r\}_{r\in\R}$.  If $A\in\cA_m$ then $A$ is more
and more smoothing if $m\rightarrow -\infty$ and the opposite as
$m\rightarrow +\infty$. We will say that $A$ is of {\em order $m$} if
$A \in \cA_m$.
    \begin{definition}
    \label{smoothing}
     We say  that $S\in\cL(\cH^{+\infty},\cH^{-\infty} )$ is $N$-smoothing if  $\forall \kappa \in \R$, it  can be extended  to an operator in   $\cL(\cH^{\kappa}, \cH^{\kappa+N})$.  When this is true  for  every $N\geq 0$, we say that $S$  is a smoothing  operator.
    \end{definition}
    The first set of assumptions concerns the properties of  $\cA_m$:\\
    
     \noindent
     {\bf Assumption I:} 
     \begin{itemize}
      \item[(i)] For each $m\in \R$, $K_0^m\in\cA_m$; in particular
        $K_0$ is an operator of order one.
\item[(ii)]  For each $m\in\R$, $\cA_m$ is a Fr\'echet space for a family of semi-norms $\{ \wp^m_j \}_{j\geq 1}$  such that the embedding $\cA_m\hookrightarrow  \bigcap_{s \in \R} \cL(\cH^s,\cH^{s-m})$ is continuous.\\
 If $m^\prime \leq m$ then $\cA_{m^\prime}\subseteq\cA_m$ with a continuous embedding.

    \item[(iii)]    $\cA$ is a graded algebra,  i.e $\forall m,n\in \R$:  if $A\in \cA_m$   and $B\in\cA_n$ then $A B\in\cA_{m+n}$ 
     and the map $(A,B)\mapsto AB$ is continuous from 
     $\cA_{m}\times\cA_{n}$ into $\cA_{m+n}$.
     \item[(iv)] $\cA$ is a graded Lie-algebra\footnote{This property will impose the choice of the semi-norms $\{ \wp^m_j \}_{j\geq 1}$. We will see in the examples that the natural  choice $(\Vert \cdot\Vert_{m,s})_{s\geq 0}$ has to be refined.  } : if $A\in \cA_m$   and $B\in\cA_n$ then the commutator $[A,B]\in\cA_{m+n-1}$  and the map $(A,B)\mapsto [A,B]$
      is continuous from $\cA_{m}\times\cA_{n}$ into $\cA_{m+n-1}$.
      
     \item[(v)] $\cA$ is closed under  perturbation by smoothing operators in the following sense:
     let $A$ be a linear map: $\cH^{+\infty}\rightarrow\cH^{-\infty}$. If   there  exists $m\in\R$ such that  for  every $N>0$  we  have  a   decomposition        $A= A^{(N)}+ S^{(N)}$,  with $A^{(N)}  \in\cA_m$  and $S^{(N)}$  is  $N$-smoothing,  then
      $A\in\cA_m$. 
       \item[(vi)] If $A \in \cA_m$ then also the adjoint operator $A^* \in \cA_m$.
  The  duality here is defined by the scalar product 
      $\lan\cdot, \cdot\ran$ of $\cH=\cH^0$. The adjoint $A^*$ is defined by 
      $\lan u, Av\ran = \lan A^*u, v\ran$  for $u,v\in\cH^\infty$  and extended by continuity.  
     \end{itemize}
    
     It is well known that classes of pseudo-differential operators
     satisfy these properties, provided one chooses for $K_0$ a
     suitable operator of the right order (see e.g. \cite{ho}).\\ In
     \cite{guil} V. Guillemin has introduced abstract
     pseudo-differential algebras, called generalized Weyl
     algebras. For his purpose \cite{guil} needs different properties
     than ours, but obviously there is an overlap with our presentation.
\begin{remark}
\label{rem:control}
One has that $ \forall A\in\cA_m$, $\forall B\in\cA_n$
\begin{align}
\label{est.1}
\forall m, s  \quad \exists N \ s.t.\  &\norm{A}_{m,s} \leq C_1 \,
\wp^m_N(A)  \ ,
\\
\label{est.2}
\forall m, n, j \quad \exists N \ s.t.\  &\wp^{m+n}_j(AB) \leq C_2 \, \wp^m_N(A)
\, \wp^n_N(B)  \ , 
\\
\label{est.3}
\forall m, n, j\quad  \exists N \ s.t.\   &\wp^{m+n-1}_j( [A, B] ) \leq C_3 \, \wp^m_N(A) \, \wp^n_N(B)  \ , 
\end{align}
for some positive
constants $C_1(s,m)$, $C_2(m,n,j)$, $C_3(m,n,j)$.
\end{remark}

For $\Omega\subset \R^d$ and $\cF$ a Fr\'echet space, 
we will denote by $C_b^m(\Omega, \cF)$ the space of
$C^m$ maps $f: \Omega\ni x\mapsto f(x)\in\cF$, such that,  for every
seminorm $\norm{\cdot}_j$ of $\cF$ one has
       \begin{equation}
       \label{star}
       \sup_{x\in\Omega}\Vert\partial_x^\alpha f(x)\Vert_{j} < +\infty
       \ , \quad \forall \alpha\in \N^d\  :\ \left|\alpha\right|\leq m  \ .
       \end{equation}
If \eqref{star} is true
$\forall m$, we say $f \in C^\infty_b(\Omega, \cF)$.

The next property needed is the following Egorov property, also well
known for pseudo-differential operators.  \\

     \noindent
       {\bf Assumption II:}  
       For any $A\in\cA_m$ and $\tau \in \R$,   the map
         $\tau\mapsto A(\tau):={\rm e}^{\im \tau K_0}\, A \, {\rm e}^{-\im \tau
         K_0}\in C^0_b(\R, \cA_m)$.

\vspace{1.5em}

       \begin{remark}
  \label{ex.ii}
From Assumption II one has that, for any $B\in\cA_{n}$, for any
$\ell\in\N$, ${\rm ad}_{A(s)}^\ell(B)\in C_b^0(]-T, T[,
    \cA_{n+(m-1)\ell})$, $\forall T>0$. Here  ${\rm ad}_{A}(B):=\im [A,B]$.
\end{remark}

       Remark that {\bf Assumption II}  is a quantum property for the time evolution of observables. Practically it follows from  the time evolution 
        of classical observables  (Hamilton equation) if  some classes of symbols are  preserved under the classical flows. Indeed one might replace {\bf Assumption II}  by   a weaker one (see Appendix   \ref{AET}).

\subsection{Perturbations of systems of order larger than 1}       
   Now  we state  our  spectral assumption on $K_0$:\\
   
   \noindent
        {\bf Assumption A :}
        $K_0$  has an entire  discrete spectrum such that 
\begin{equation}
\label{specK0}
{\rm spec}(K_0) \subseteq \N +\lambda
\end{equation}  
 for some $\lambda > 0$.\\
 
Our second spectral assumption is essentially that the unperturbed
operator $H_0$ is a function of $K_0$. To state it precisely we need the
following definition

  \begin{definition}
  \label{def:cl.sy}
  A function $f\in C^\infty(\R)$  will be said to be a {\em   classical symbol of order $\rho$} (at $+\infty$)  if there  exist real numbers $\{c_j\}_{j\geq 0}$ s.t. 
  $c_0\geq 0$ and 
  for all $k\geq 1$, all $N\geq 1$, there exists $C_{k,N}$ s.t. 
  $$
  \abs{\frac{d^k}{dx^k}\bigl(f(x)-\sum_{0\leq j\leq N-1}c_jx^{\rho-j}\bigr)} \leq C_{k,N}\vert x^{\rho-N-k}\vert, \;\; \forall x\geq 1.
  $$ We will denote by $S^\rho$ the space of classical symbols of
  order $\rho$.\\ We shall say that $f$ is an {\em elliptic classical
    symbol of order $\rho$} if $f$ is real and $c_0>0$. We shall write
  $f \in S^\rho_+$.\\ 
We shall say that $f$ is a {\em classical symbol
    of order $- \infty$} if $f \in S^m$ $\forall m <0$. We shall write
  $f \in S^{- \infty}$. 
    \end{definition}
Some standard properties of classical symbols are recalled in Appendix \ref{sec:tec}. 
We assume that\\

\noindent {\bf Assumption B:} There exists an elliptic classical
symbol $f$ of order $\mu>1$, such that
  \begin{equation}\label{H0}
H_0=f(K_0) \ . 
\end{equation}  
\vspace{.5em}\\
We will prove (see Lemma \ref{ref:fc}) that \eqref{H0} implies  $H_0 \in
\cA_{\mu}$, i.e. $H_0$ is an operator of order $\mu>1$. 
   
We come back to the Schr\"odinger equation defined by the time
dependent Hamiltonian $H(t) := H_0 + V(t)$ (see \eqref{LS}).  When the
solution $\psi(t)$ exists globally in time, we define the
Schr\"odinger propagator $\cU(t,s)$, generated by \eqref{LS}, such
that \beq
   \label{propag}
    \psi(t)=\cU(t, s)\psi \ , \quad
    \cU(s,s)=\uno
    \eeq
    We are ready to state our main result on systems with increasing gaps:
\begin{theorem}\label{thm:main}
Assume that $\cA$ is a graded algebra as defined in Section
\ref{sec:alg} and that $K_0$, $H_0$ satisfy assumptions A and
B. Furthermore assume that the perturbation $V(t)$ with domain
$\cH^\infty$ is symmetric for every $t\in\R$ and satisfies
    \begin{equation}
    \label{V.cond}
    V\in C_b^\infty(\R, \cA_\r) \ , \qquad \mbox{ with } \ \ \r<
    \mu\ .  
    \end{equation}
     Then $H(t)=H_0 + V(t)$  generates a propagator $\cU(t,s)$  s.t.  $\cU(t,s)\in\cL(\cH^r)$  $\, \forall r\in\R$. \\
     Moreover
     for any $r > 0$ and any $\epsilon >0$  there exists $C_{r,\epsilon}>0$ such that
  \beq\label{pr1}
     \Vert \cU(t,s)\psi\Vert_r\leq  C_{r,\epsilon}\la t-s\ra^\epsilon\Vert\psi\Vert_r,\qquad \forall t, s\in\R.
     \eeq
      \end{theorem}
This result extends a result by Nenciu \cite{nen} for bounded
perturbations ($\r=0$). Furthermore in \cite{MaRo} two of us had
already extended Nenciu's result to unbounded perturbations with the constraint $\r< \min( \mu-1, 1)$. The main point is that we
add here a stronger spectral assumption: essentially the spectrum of
$H_0$ is $f(\N+\lambda)$ for some smooth function $f$ (see Assumptions A and
B).

As a final remark, we note that Theorem \ref{thm:main} gives also a
proof of the existence and of some properties of the propagator
$\cU(t,s)$, which in the framework of Theorem \ref{thm:main} are not
obvious.


\subsection{Applications (i)}
           
      \paragraph*{Zoll manifolds.}
      Recall that a Zoll manifold is a compact Riemannian manifold
      $(M, g)$ such that all the geodesic curves have the same period
      $T:=2\pi$. For example the $d$-dimensional sphere $\S^d$ is a Zoll manifold.  
      We
      denote by $\triangle_g$ the positive Laplace-Beltrami operator
      on $M$ and by $H^r(M) = {\rm Dom}(1+\triangle_g)^{r/2}$, $r\geq 0$, the usual scale of Sobolev spaces.  Finally we denote by
      $S_{\rm cl}^m(M)$ the space of classical real valued symbols of
      order $m\in \R$ on the cotangent $T^*(M)$ of $M$ (see
      H\"ormander \cite{ho} for more details).
      \begin{definition}
        \label{pseudo}
We say that $A\in\cA_m$ if it is a pseudodifferential operator (in the sense of H\"ormander \cite{ho}) with
symbol of class $S^m_{\rm cl}(M)$.
      \end{definition}
 In this case the operator $K_0$ is a
      perturbation of order $-1$ of $\sqrt{\triangle_g}$ (see
      Sect. \ref{zoll.app}), and the norms $\left\|\psi\right\|_r$
      coincide with the standard Sobolev norms.

      \begin{corollary}[Zoll manifolds] 
      \label{Zoll} Let $V(t)$ be a symmetric pseudo-differential
      operator  of order $\r<2$  on $M$   such that its symbol
      $v\in C^\infty_b(\R; S^\rho_{\rm cl}(M))$.      
         Then  the propagator $\cU(t,s)$   generated by $H(t)= \triangle_g +  V(t)$ exists and satisfies \eqref{pr1}.
      \end{corollary}
    
      \paragraph*{Anharmonic oscillators on $\R$.}
      The  second application concerns one dimensional quantum anharmonic oscillators 
\begin{equation}
\label{an.os}
\im \partial_t \psi = H_{k,l} \psi + V(t) \psi \ , \qquad x \in \R \ ,
\end{equation}      
      where $H_{k,l}$ is the     
       one  degree  of freedom Hamiltonian
  \begin{equation}
\label{Hkl}H_{k,l}:= D_x^{2l} + ax^{2k} \ , \qquad k, l \in \N \ ,
\ \ \ k+l \geq 3 \ , \ \ \ a >0  \ .
\end{equation}  
      Here   $D_x:=\im^{-1}\partial_x$. 
       It is well known that $H_{k,\ell}$ is essentially self-adjoint in $L^2(\R)$ \cite{hero}. \\
            Define the Sobolev spaces $\cH^r: ={\rm Dom}(H_{k,l}^{\frac{k+l}{2kl}r})$  for $r\geq 0$.
            We 
            define now   suitable operator classes  for the perturbation. 
            Denote 
            $$\tk_0(x,\xi) := (1+x^{2k}+\xi^{2l})^{\frac{k+l}{2k l}} \ . $$
\begin{definition}
\label{symbol.ao}
A function $f$ will be called a symbol of order  $\r\in\R$ if  $f \in C^\infty(\R_x \times \R_\xi)$ and 
               $\forall \alpha, \beta \in \N$, there exists $C_{\alpha, \beta} >0$ s.t. 
\beq
\label{es.7}
             \vert \partial_x^\alpha \, \partial_\xi^\beta f( x,\xi)\vert \leq C_{\alpha,\beta} \ \tk_0(x,\xi)^{\r-\frac{k\beta +l\alpha}{k+l}}  \ . 
             \eeq
             We will write $f \in S^\r_{\ao}$.
\end{definition}
As usual to a symbol $f \in S^\rho_\ao$ we associate the operator
$f(x, D_x)$ which is obtained by standard Weyl quantization (see formula
\eqref{weyl} below).
\begin{definition}
\label{pseudo.an}
  We say that $F\in \cA_\rho$ if it is a pseudodifferential operator
  with symbol of class $S^\r_{\ao}$, i.e., if there
  exist $f \in S^\r_{\ao}$ and $S$ smoothing (in the sense of
  Definition \ref{smoothing}) such that $F = f(x, D_x) + S$.
\end{definition}
In this case the seminorms are defined by
$$
\wp^\rho_j(F):=\sum_{|\alpha|+|\beta|\leq j}C_{\alpha\beta}\ ,
$$
with $C_{\alpha\beta}$ the smallest constants s.t. eq. \eqref{es.7} holds.
If a symbol $f$ depends on additional parameters (e.g. it is time dependent), we ask that the constants $C_{\alpha, \beta}$ are uniform w.r.t. such parameters.

\begin{remark}
With this definition of symbols, one has  $x \in S^{\frac{l}{k+l}}_\ao$, $\xi \in S^{\frac{k}{k+l}}_\ao$, $x^{2k} + \xi^{2l} \in S^{\frac{2kl}{k+l}}_\ao$, 
$\tk_0(x,\xi) \in S^{1}_\ao$.
\end{remark}

          We get the following:   

       \begin{corollary}[1-D anharmonic oscillators] 
       \label{Anha}
       Consider equation \eqref{an.os} with the assumption \eqref{Hkl}.
       Assume also that $V\in C^\infty_b(\R;\cA_\rho)$ with $\rho< \frac{2kl}{k+l}$.   Then   the propagator $\cU(t,s)$   generated by $H(t)= H_{k,l} + V(t)$ is well  defined  and   satisfies \eqref{pr1}.
       \end{corollary}

\noindent An  example of admissible perturbation is   $\displaystyle{V(t,x,\xi) = \sum_{l \alpha+k\beta <2kl }a_{\alpha,\beta}(t)x^\alpha\xi^\beta}$
        with $a_{\alpha,\beta}\in C_b^\infty(\R, \R)$. In particular if we choose $H_0=-\frac{d^2}{dx^2}+x^4$, we can consider unbounded perturbations of the form $x^3 g(t)$ and of course also $x g(t)$  with $g \in C_b^\infty(\R, \R)$.

\begin{remark}
          \label{red.2}
Our class of perturbations contains quite general
pseudodifferential operators, however it is easy to see that
multiplication operators (i.e. operators independent of $\partial_x$)
must be polynomials in $x$ with coefficients which are possibly time
dependent.

In the similar problem of reducibility more general classes of
perturbations have been treated in \cite{Bam17b}. We did not try
to push the result in that direction. This is probably non trivial in
an abstract framework like the one we are using here.
\end{remark}

\begin{remark}
  \label{maxim}
We think that our method should also allow to deal with some
perturbations of the same order as the main term.
For example it should be treatable the case
where $V$ is a quasihomogeneous polynomial of maximal order fulfilling
some sign condition (more or less as in Theorem 2.12 of
\cite{Bam16I}).
\end{remark}

\subsection{Perturbations of systems of order 1}

In order to deal with perturbations of operators of order 1 we have to
restrict to the case where the dependence of the perturbation on time
is quasiperiodic.

Let $\cA:=\cup_{m\in\R}\cA_m$ be a graded Lie algebra satisfying {\bf
  Assumption I } with a reference operator $K_0$.  \\ Let $K_1,K_2,
\cdots, K_{d}$ be $d$ self-adjoint positive operators such that
$K_j\in\cA_1$, $\forall 1 \leq j \leq d$. Assume the following modified
Assumption II:\\

    \noindent 
        {\bf Assumption II$^\prime$:}  
              \begin{itemize}
       \item[(i)] $[K_j,K_\ell]=0$ for any $0\leq j,\ell\leq d$. 
       \item[(ii)] Denote $K=(K_1,\cdots, K_{d})$ and for
         $\tau\in\R^{d}$, $\displaystyle{\tau\cdot K:=\sum_{1\leq
           j\leq d} \tau_j K_j}$.  \\ Then for any $A\in\cA_m$,  the map 
         $\tau\mapsto A(\tau):={\rm e}^{\im \tau\cdot K}A{\rm e}^{-\im \tau\cdot K}\in
         C^{\infty}_b(\R^d ;\cA_m)$.
       \end{itemize}

\begin{remark}
                \label{ex.iip}
For  any $B\in\cA_{n}$, for any $\ell\in\N$, one has  ${\rm
  ad}_{A(s)}^\ell(B)\in$  $C^{\infty}_b(\R^d ;\cA_{n + \ell(m-1)})$.
               \end{remark}
       
       We also adapt our spectral conditions:\\
 
         \noindent
        {\bf Assumption A$^\prime$:} 
          $K=( K_1,\cdots, K_{d})$  has an entire  joint  spectrum, ${\rm spec}(K) \subseteq \N^{d}+\lambda$  for some $\lambda\in \R^{d}$,  $\lambda \geq 0$.\\
          
          \noindent  {\bf Assumption B$^\prime$:} 
There exist $\left\{\nu_j\right\}_{j=1}^d $, $\nu_j>0$  s.t.  
\begin{align}
  \label{II.1}
  H_0= \sum_{1\leq j\leq d}\nu_j K_j \  ,
  \\
  K_0=H_0 \ .
\end{align}

In order to fix ideas one can think of the case of Harmonic
oscillators, in which $K_j=-\partial_j^2 +x_j^2$, $1\leq j\leq
d$.

\begin{remark}
  \label{pos}
Since the operators $K_j$ are positive, the norm $\norma{.}_r$ defined
using the operator $K_0$ is equivalent to the norm defined using the
operator $K_0':=\sum_{j=1}^dK_j$. 
\end{remark}

We consider both the case where $$\nu:=(\nu_1,...,\nu_d)$$ is
resonant and the case where it is nonresonant.
%
To state the arithmetical assumptions on $\nu$, we first recall the
following well known lemma whose scheme of proof will be recalled in the
Appendix \ref{tilde}.

\begin{lemma}
  \label{lem.gio}
There exists $\tilde d\leq d $, a vector $\tilde \nu \in \R^{\tilde
  d}$ with components independent over the rationals, and vectors
$\bv_j\in \Z^d$, $j=1,...,\tilde d$ such that
\begin{equation}
  \label{vtilde}
\nu = \sum_{j=1}^{\tilde d} \tilde \nu_j \, \bv_j \ .
\end{equation}
\end{lemma}


%

\begin{remark}
  \label{deg}
  For example\\
(i) if $\nu$ is nonresonant, then $\tilde \nu = \nu$ and $\bv_j =
  \be_j$, the standard basis of $\R^d$;\\
   (ii) if  $\nu$ is completely resonant then $\tilde d =1$; e.g. if $ \nu=(1, \ldots, 1)$, then $\tilde \nu_1 =1$, $ \bv_1 = (1, \ldots, 1) $. 
\end{remark}

\begin{theorem}\label{thm:hosc}
Assume that $V(t)=W(\omega t)$ with $W\in C^\infty_b(\T^n, \cA_\r)$ a
quasi-periodic operator of order $\rho<1$.  Assume furthermore that 
$(\tilde \nu,\omega) \in\R^{\tilde d+n}$ is a Diophantine
vector, namely that there exist $\gamma>0,$ and $ \kappa\in\R$ s.t., 
\begin{equation}
\label{non.res3}
\abs{\omega \cdot k + \tilde \nu \cdot \ell } \geq
\frac{\gamma}{(|\ell\vert +\vert k|)^\kappa}\ , \quad
0\not =(k,\ell)\in\Z^{n+\tilde d}\ .
\end{equation}
Then the propagator $\cU(t,s)$ generated by $H(t)= \nu\cdot K +
W(\omega t)$ exists and satisfies
\eqref{pr1}.
\end{theorem}

\begin{remark}
The  vector $\tilde \nu$ is defined up to linear combinations with integer coefficients; clearly  condition \eqref{non.res3} does not depend on the choice of $\tilde \nu$.
\end{remark}

\begin{remark}
We recall that Diophantine vectors form a subset of $\R^{n+\tilde d}$ of full
measure if $\kappa>n+\tilde d-1$.
\end{remark}

\subsection{Applications (ii)}
\label{app.sec.2}
\paragraph{Relativistic Schr\"odinger equation on Zoll manifolds.}
We consider the reduced Dirac equation on a Zoll manifold $M$ with
mass $\mu>0$
$$
\im \partial_t \psi = \sqrt{\triangle_g +\mu} \ \psi + V(\omega t, x, D_x)\psi \ , \qquad t\in\R,\ x\in M \ .
$$ As in the case of the Schr\"odinger equation on Zoll manifolds,
$\cA_\r$ is the class of pseudodifferential operators with symbols in
$S^\r_{\rm cl}(M)$ (see Definition \ref{pseudo}).

In this case  $V$  is assumed to be quasi-periodic in time.

\begin{corollary}[Relativistic Schr\"odinger equation on Zoll manifolds] 
\label{cor:rs.zoll}
Assume that $V(t) = W(\omega t)$ with $W  \in C^\infty(\T^n, \cA_\rho)$ with $\rho < 1$. Assume
furthermore that the non resonance condition
\begin{equation}
  \label{non.res.re}
\left|\omega\cdot k+m\right|\geq \frac{\gamma}{1+|k|^\kappa}\ ,\quad
\forall 0\not=k\in\Z^n\ ,\quad \forall m\in\Z
\end{equation}
holds for some $\gamma>0$ and $\kappa$. Then the propagator $\cU(t,s)$
generated by $H(t)= \sqrt{\triangle_g +\mu} + W(\omega t)$ exists and satisfies
\eqref{pr1}.
\end{corollary}

\paragraph{Harmonic oscillator in $\R^d$.} Consider the quantum Harmonic
oscillator
\begin{align}
  \label{qhosc.1}
  \im\partial_t\psi=H_\nu\psi+V(t)\psi\ ,\quad x\in\R^d
  \\
  \label{qhosc.2}
H_\nu:=-\Delta +\sum_{j=1}^d\nu_j^2 x_j^2\ ,\qquad
V(t)=W(\omega t,x,D_x)\ .
\end{align}
Here $W$ is the Weyl quantization of a symbol belonging to the following class
\begin{definition}
\label{symbol.ao2}
A function $f$ will be called a symbol of order  $\r\in\R$ if  $f \in C^\infty(\R^d_x \times \R^d_\xi)$ and 
               $\forall \alpha, \beta \in \N^d$, there exists $C_{\alpha, \beta} >0$ s.t. 
\beq
\label{es.71}
             \vert \partial_x^\alpha \, \partial_\xi^\beta f( x,\xi)\vert \leq C_{\alpha,\beta} \ (1 + |x|^2 + |\xi|^2)^{\r-\frac{|\beta| + |\alpha|}{2}}  \ . 
             \eeq
             We will write $f \in S^\r_{\hos}$.
\end{definition}

\noindent The class \eqref{es.71} is the extension to higher dimensions  of the class used in the anharmonic oscillators (see Definition \ref{symbol.ao}) and with  $k=l=1$.

\begin{remark} With our numerology, the symbol of the harmonic oscillator is of order 1, 
$|\xi|^2 + \sum_j \nu_j^2 x_j^2 \in S^1_\hos$, and not of order 2 as typically in the literature.
\end{remark}

The classes $\cA_m$ are defined as in Definition \ref{pseudo.an}, with
symbols in the class $S^m_\hos$.

 \begin{corollary}
   \label{cor:hosc}
Assume that $\nu$ is such that $\tilde \nu$ fulfills \eqref{non.res3}, and 
that $W\in C^\infty(\T^n;\cA_\rho)$ with $\rho<1$. Then the propagator
$\cU(t,s)$ of $H(t) = H_\nu + W(\omega t)$ exists and fulfills \eqref{pr1}.
 \end{corollary}
 
 Remark that after a trivial rescaling of the spatial variables, $H_\nu = \sum_{j=1}^d \nu_j (-\partial_j^2 + x_j^2)$, thus the corollary is a trivial application of Theorem \ref{thm:hosc}. 

 \begin{remark}
   \label{resonant}
In the completely resonant case
$$
H_{(1,...,1)}=-\Delta+|x|^2\ ,
$$
one has $\tilde \nu=1$ and the set of the $\omega's$ for which
\eqref{non.res3} is fulfilled has full measure provided $\kappa>n$.
 \end{remark}
 \begin{remark}
   \label{del.gr}
We note that in the resonant case there have been exhibited examples of polynomial growths of the Sobolev norms. In particular see \cite{del} and \cite{BGMR1} for periodic in time perturbations; of course in
such examples the frequency $\omega$ does not fulfill
\eqref{non.res3}.
Finally we recall also  \cite{buni}, where some some  random in time perturbations are considered.
 \end{remark}
 \section{Proofs of the abstract theorems}

\subsection{Scheme of the proof}\label{sec:sketch}   

As explained in the introduction, the main step of the proof consists
in proving a theorem conjugating the
original Hamiltonian to a Hamiltonian of the form \eqref{Z.1}; this
will be done in  Theorem \ref{thm:smooth1}. Subsequently we
will apply Theorem 1.5 of \cite{MaRo}, which essentially states that,
if $H(t)$ is such that {for some $N > -1$}
\begin{equation}
\label{maro.cond}
[ H (t), K_0] K_0^{N} \in  C^0_b(\R, \cL(\cH^r)) \ , 
\end{equation}
then $\exists C_{r,N}>0 $ such that 
\begin{equation}
\label{maro.est}
\norm{\cU(t,s) \, \psi}_r \leq C_{r, N} \,  \la t-s \ra^{\frac{r}{1+N}} \, \norm{\psi}_r \ , \qquad
\forall t, s \in \R \ . 
\end{equation}

%
%

{We come to the algorithm of conjugation of the original Hamiltonian to
\eqref{Z.1}. Before discussing it, we need to know the way a
Hamiltonian is changed by a time dependent unitary
transformation. This is the content of the following lemma.}

\begin{lemma}
\label{T.1}
Let $H(t)$ be a time dependent self-adjoint operator, and $X(t)$ be a selfadjoint family of operators.  Assume that
$\psi(t)=\e^{-\im X(t)}\vf(t)$ then
\begin{equation}
\label{1}
\im\dot \psi=H(t) \psi\ \quad \iff \quad \im\dot \vf =\tilde H(t) \vf 
\end{equation}  
where 
\begin{align}
\label{4.1.1}
\tilde H(t):= \e^{\im X(t)} \, H(t)\, \e^{-\im X(t)}
-\int_0^1 \e^{\im s X(t) } \,\dot X(t)  \, \e^{-\im
  s X(t) }  \ \di s \ .
\end{align}
\end{lemma}
\noindent This is seen by an explicit computation. For example see Lemma 3.2 of
\cite{Bam16I}. 

A further important property giving the expansion of an operator
of the form $\e^{\im
  X(t)} \, A\, \e^{-\im X(t)}$ in operators of decreasing order is
stated in the following lemma.

\begin{lemma}
\label{comX}
Let $X\in\cA_{\rho}$ with $\rho< 1$ be a symmetric operator. Let $A\in\cA_{m}$ with
$m \in \R$. Then $X$ is selfadjoint and for any $M\geq 1$ we have
\begin{equation}
\label{expansion}
\e^{\im \tau X}\, A \,  \e^{-\im\tau X } = \sum_{\ell=0}^{M}\frac{\tau^{\ell}}{\ell!}{\rm ad}_X^\ell(A) + R_M(\tau, X,A) \ , \qquad \forall \tau \in \R \ , 
\end{equation}
 where  $R_M(\tau, X,A)\in\cA_{m -(M+1)(1-\rho)}$.  
\\
In particular ${\rm ad}_X^{\ell}(A)\in\cA_{m-\ell(1-\rho)}$ and  
${\rm e}^{\im \tau X}A{\rm e}^{-\im \tau X}\in \cA_{m}$, $\forall \tau
\in \R$.
\end{lemma}
The proof will be given in Sect. \ref{flo.lem}.

We describe now the algorithm which will lead to the smoothing Theorem
\ref{thm:smooth1}; the proof is slightly different according to the
set of assumptions one chooses. We start by discussing it under the
assumptions of Theorem \ref{thm:main}, namely Assumption A and B.
Subsequently we will discuss the changes needed to deal with Theorem
\ref{thm:hosc}.

We look for a change of variables of the form $\vf = \e^{ \im X_1(t)}
\psi$ where $X_1(t)\in\cA_{\rho-\mu + 1 }$ is a self-adjoint operator
which, due to the assumption $\rho < \mu$, has 
order smaller then one. Then $\vf$ fulfills the Schr\"odinger equation
$\im \dot \vf = H^+(t) \vf$ with
\begin{align*}
\label{4.1.1}
H^{+}(t)&:= \e^{\im X_1(t)} \, H(t)\, \e^{-\im X_1(t)}
-\int_0^1 \e^{\im s X_1(t) } \,\dot X_1(t)  \, \e^{-\im
  s X_1(t) }  \ \di s  \\
  &= H_0 + \im [ X_1(t), H_0]  + V(t)+\im [X_1(t),V(t)]
  -\frac{1}{2}[X_1(t), [ X_1(t), H_0] ] +\cdots\\
   &\qquad -\int_0^1 \e^{\im s X_1(t) } \,\dot X_1(t)  \, \e^{-\im
  s X_1(t) }  \ \di s .
\end{align*} 
In view of the properties of the graded algebra we have $[X_1,V]\in \cA_{2\r
  -\mu}$, $[X_1, [ X_1, H_0]]\in \cA_{2\r -\mu}$ (Assumption I (iv)) and
$\e^{\im s X_1(t) } \,\dot X_1(t) \, \e^{-\im s X_1(t) }\in
\cA_{\r-\mu+1}$ (Lemma \ref{comX}),  therefore one has
\begin{equation}
\label{4.1.a}
H^{+}(t)= H_0 + \im [ X_1(t), H_0]  + V(t) + V^+_1(t) \ ,
\end{equation} 
 with $V_1^+(t) \in C^\infty_b(\R, \cA_{\min(\r-\mu+1,2\r
  -\mu)})$.
 \\ 
  Now  we look for $X_1(t)$ s.t.
\begin{equation}
\label{hom.4}
\im [H_0,X_1(t)] =V(t) - \lan V(t)\ran  \ , 
\end{equation}
{where $ \lan V(t)\ran $ is the average over $\tau$ of $e^{\im\tau
  K_0}V(t)e^{-\im\tau K_0}$ (see \eqref{average}), which in particular
commutes with {$K_0$}.  We will verify in
Lemma \ref{hom} that there exists $X_1$ s.t.
$$
\im [H_0,X_1(t)] -V(t) + \lan V(t)\ran\in \cA_{\rho-1}\ . 
$$
Therefore using such a $X_1$ to generate a unitary transformation, we get}
  \begin{align}
\label{4.1.3}
H^{+}(t)&:= H_0 + \la V(t) \ra + V^+(t) \ , 
\end{align} 
 where $ V^+(t) \in C^\infty_b(\R, \cA_{\r - \delta})$
  with 
\begin{equation}
\label{delta}
\delta := \min \left( 1, \mu-1, \mu - \rho \right)>0 \ .
\end{equation}
Therefore  $V^+(t)$ is a perturbation of order  lower than $V(t)$.  
Furthermore $\la V(t) \ra$ commutes with $K_0$.

Iterating this procedure  we will establish an "almost" reducibility
result that will be stated and proved in Subsect. \ref{sec:iterative}.

{Then, using Theorem 1.5 {of \cite{MaRo}}, we immediately get Theorem
  \ref{thm:main}.}

In the case where $H_0\in\cA_1$ the procedure has to be slightly
modified since in this case $X_1$ and therefore $\dot X_1$ has the
same order as $V$ and thus it cannot be considered as a remainder when
analyzing $H^+$. In this case one rewrites 
\begin{align*}
\label{4.1..2}
H^+(t)&=H_0 + \im [ X_1(t), H_0]  + V(t)
\\
&+\im [X_1(t),V(t)]
  -\frac{1}{2}[X_1(t), [ X_1(t), H_0]] +\cdots\\
   &-\dot X_1 - \int_0^1 \left(\im \, s \, [X_1(t),\dot X_1(t)]+....\right) \di s ,
\end{align*}
so that eq. \eqref{4.1.a} is substituted by
\begin{equation}
\label{4.1.b}
H^{+}(t)= H_0 + \im [ X(t), H_0]  + V(t) -\dot X_1(t)+ V^+(t) \ ,
\end{equation}
with $V^+\in\cA_{\rho- \delta_*}$, 
\begin{equation}
\label{delta*}
\delta_* := 1 - \rho > 0  \ , 
\end{equation}
so again it is more regular than $V(t)$. Thus one is led to consider the new
homological equation
\begin{equation}
\label{hom.5}
\im [H_0,X_1(t)]+\dot X_1(t) =V(t) - \lan V(t)\ran \ ,
\end{equation}
{where $\lan V(t)\ran$ has to commute with $K_0$. 
In order to be able to solve such an equation we restrict to the case
of $V(t)$ quasiperiodic in $t$ and, as explained in the introduction,
we develop a procedure based on a suitable Fourier expansion to
construct $X_1$ and $\langle V(t)\rangle$. The details are given in 
 Lemma \ref{lem:hom2} which will ensure that
such a homological equation has a smooth solution and thus the
procedure is well defined also in the case of order 1.
}
   
\subsection{A couple of lemmas on flows}\label{flo.lem}

\begin{lemma}
\label{MR}
$(i)$ Let  $X\in\cA_1$ be symmetric w.r.t. the scalar product  of $\cH^0$. Then  $X$ has a unique  self-adjoint  extension 
 and $\e^{-\im \tau X} \in \cL(\cH^r)$ $\, \forall r \geq 0$ and  $\forall \tau\in\R$. Furthermore $\e^{-\im \tau X}$ is an isometry in $\cH^0$.\\
$(ii)$ Assume that $X(t)$ is a family of symmetric operators in $\cA_1$ s.t. 
 \begin{equation}
 \label{X(t).est}
 \sup_{t \in \R} \wp^1_j(X(t)) < \infty \ , \quad 
\forall j \geq 1 \ . 
 \end{equation}
 Then there exist $c_r, C_r >0$ s.t.
 \begin{equation}
 \label{X(t).est0}
 c_r \norm{\psi}_r \leq \norm{\e^{-\im \tau X(t)} \psi}_r \leq 
 C_r \norm{\psi}_r \ , \qquad \forall t \in \R \ , \quad \forall \tau \in [0,1]  \ .
 \end{equation}
\end{lemma}
\begin{proof} (i)  From the properties  of the  algebra  $\cA$ we have that $ X K_0^{-1}$ and  $[X,K_0] K_0^{-1}$ are of order $0$.  Thus by definition these operators belong to  $\cL(\cH^r)$ $\, \forall r \in \R$. Then the result follows from  Theorem 1.2 of \cite{MaRo}. \\
(ii) By item (i), for any $t \in \R$ and $\tau \in [0,1]$ the operator $\e^{-\im \tau X(t)}$ is an isometry in $\cH^0$, therefore
$$
\norm{\e^{-\im \tau X(t)}\psi}_{r} = \norm{\e^{\im \tau X(t)} \, K_0^{r} \, \e^{-\im \tau X(t)} \psi}_0 \ . 
$$
Then  we have
\begin{align}
\notag
\e^{\im \tau X(t)} \, & K_0^{r} \, \e^{-\im \tau X(t)} \psi 
 = K_0^r \psi + \im \int_0^\tau \e^{\im \tau_1 X(t)} \,[X(t),  K_0^{r}]  \, \e^{-\im \tau_1 X(t)} \psi \, \di \tau_1 \\
\label{X(t).est1}
& = K_0^r \psi + \im \int_0^\tau \e^{\im \tau_1 X(t)} \,[X(t),  K_0^{r}] K_0^{-r} \, K_0^r  \, \e^{-\im \tau_1  X(t)} \psi \, \di \tau_1 
\end{align}
By the properties of the algebra $\cA$ and \eqref{X(t).est} one has that (using \eqref{est.1}--\eqref{est.3})
$$
\sup_{t \in \R} \norm{[X(t),  K_0^{r}] K_0^{-r}}_{\cL(\cH^0)} < C_r < +\infty \ , 
$$
therefore taking the norm $\norm{\cdot}_0$ of \eqref{X(t).est1} one gets the inequality
$$
\norm{\e^{-\im \tau X(t)}\psi}_{r} \leq \norm{\psi}_r + \int_0^\tau C_r \norm{ \e^{-\im \tau_1 X(t)} \psi}_r \, \di \tau_1 \ . 
$$
Then by Gronwall we conclude that
$$
\norm{\e^{-\im \tau X(t)}\psi}_{r} \leq \e^{C_r} \norm{\psi}_r \ , \quad \forall t \in \R , \ \ \ \forall \tau \in [-1,1] \ . 
$$
This proves the majoration in \eqref{X(t).est0}. The minoration follows simply by the identity $\psi =e^{\im \tau X(t)} e^{-\im \tau X(t)}\psi$ and the majoration.
\end{proof}

\begin{proof}[Proof of Lemma \ref{comX}.]
Selfadjointness was proven in the previous lemma. Let us apply to the
l.h.s. of \eqref{expansion} the Taylor formula at $\tau=0$.  Then we
get, with $U_X(\tau) := {\rm e}^{-\im \tau X}$ and ${\rm ad}_X(A)
:=\im [X,A]$
\begin{align}
\label{tay.q}
&U_X(-\tau)\, A \, U_X(\tau)
\\
\nonumber
&=\sum_{j=0}^{M}\frac{\tau^j}{j!}{\rm ad}_X^j(A) + 
\frac{\tau^{M+1}}{M!}
\int_0^1(1-s)^{M+1} U_X(-s\tau)\, {\rm ad}_X^{M+1}(A)\, U_X( s\tau) \,\di s\ .
\end{align}
Using Assumption I (iv), we have ${\rm ad}_X^j(A)\in\cA_{m -j(1-\rho)}$. We define the remainder $R_M(\tau, X, A)$ to be the integral term in \eqref{tay.q}, which, using also Lemma \ref{MR},    belongs to $\cL(\cH^s, \cH^{s - m + (M+1)(1-\rho)})$, $\, \forall s \in \R$. Therefore the remainder $R_M(\tau, X, A)$  is $N$-smoothing  
provided $M+1 \geq \frac{N + m}{1-\rho}$. As  $M$ can be taken arbitrary large, ${\rm e}^{\im \tau X}A{\rm e}^{-\im
  \tau X}$ fulfills Assumption I (v), thus it belongs to $\cA_{m}$.
\end{proof}

\subsection{Solution of the Homological equations}\label{sol.homo.2}

\paragraph{The first homological equation.} 
As we have seen in Section
\ref{sec:sketch}, to prove Theorem \ref{thm:main}  we need to study
an homological equation of the form  
\beq
\label{hom1}
\im [H_0, X] = A-\la A\ra, 
\eeq 
where $A\in \cA_m$  and $\lan A\ran$ is the average of $A$  along the periodic flow of $K_0$:
\begin{equation}
\label{average}
\lan A\ran :=\frac{1}{2\pi}\int_0^{2\pi}A(\tau) \, \di \tau \ , \qquad A(\tau) = \e^{\im \tau K_0} \, A \, \e^{- \im \tau K_0}.
\end{equation}
 Notice that  the assumption on the spectrum of $K_0$  (see Assumption A) entails that {$\e^{2\im \pi K_0} = \e^{2 \im \pi \lambda}$, thus for any $A \in \cA$ one has  $\e^{2 \im \pi K_0} \, A  \, \e^{-2 \im \pi K_0} = A$, namely  $\tau \mapsto A(\tau)$ is $2\pi$ periodic.}

\begin{lemma}
\label{normal.form}
Let $A \in \cA_m$, $m \in \R$. Then $\la A \ra \in \cA_m$ and
\begin{equation}
\label{eq.nf}
[ K_0, \la A \ra ] = 0 \ . 
\end{equation}
\end{lemma}
\begin{proof}
$\la A \ra \in \cA_m$ is a consequence of  Assumption II. 
Identity \eqref{eq.nf} follows by a direct computation.
\end{proof}

\begin{lemma}\label{hom} 
\noindent
(i) Let $A \in \cA_m$, $ m \in \R$. Then
\beq
\label{sol1}
Y = \frac{1}{2\pi}\int_0^{2\pi}\tau \, (A-\la A\ra)(\tau)\, \di  \tau  
\eeq solves the homological  equation
\beq\label{hom2}
\im[K_0, Y] = A-\la A\ra.
\eeq 
Further $Y\in \cA_m$ and if $A$ is symmetric, so is $Y$.\\
(ii) Choose  $R>0$ such that $f^\prime(x) \geq 1$ if $x\geq R$ and $\eta\in C^\infty(\R)$ such that 
$\eta(x)=1$ if $x\in[0,R]$, $\eta(x)=0$ if $x\geq R+1$. 
Define  
\begin{equation}
\label{sol2}
X:= \left(1-\eta(K_0)\right) \, \left(f^\prime(K_0)\right)^{-1} \, Y \ , 
\end{equation}
with $Y$ as in \eqref{sol1}. Then $X\in\cA_{m -\mu+1}$, is symmetric
provided $A$ is symmetric and solves \eqref{hom1} modulo an
error term in $\cA_{m-1}$. More precisely \beq\label{hom3} \im [H_0, X] = A-\la A\ra
+ \cA_{m-1} \ .  \eeq
\end{lemma}
We note for the sequel that  if $A \in \cA_m$ then   $X \in \cA_{m
  -(\mu-1)}$, namely   we have a gain of $\mu-1>0$ in the smoothing order.
\begin{proof}
{Assertion (i) is proved by  integration by parts using that $A(\tau)$ is $2\pi$-periodic.}\\
To prove (ii), first remark that by Assumption B and Lemma
\ref{rem:class.symb}, $f' \in S^{\mu-1}_+$,   thus it is
{different from zero} provided $ x \geq R$ is large enough. 
It follows that the function $\displaystyle{x \mapsto \frac{1-\eta(x)}{f'(x)} \in S^{-\mu+1}}$. Therefore, by  Lemma \ref{ref:fc}, the operator $\left(1-\eta(K_0)\right) \, (f'(K_0))^{-1} \in \cA_{-\mu+1}$. Finally  since $Y \in \cA_m$, it follows that $X \in  \cA_{m- \mu+1}$. \\
We show now that $X$ solves \eqref{hom3}. 
This is a consequence of  the commutator expansion Lemma.
Indeed fix $N \geq 2$, then  by Lemma \ref{lem:gerard} one has 
\begin{align*}
[H_0, X ] & = [f(K_0), X] = f'(K_0) [K_0, X]  + \sum_{2 \leq j \leq N}\frac{1}{j!} f^{(j)}(K_0) {\rm ad}_{K_0}^j (X) +  R_{N+1}(f, X) 
\end{align*}
with $R_{N+1}(f, X) \in \cA_{m -\mu+1+ [\mu]-N}\subset \cA_{m-1}$. \\
By Lemma \ref{rem:class.symb} and Assumption I, for any integer $j \geq 2$ one has that $f^{(j)}(K_0)\,  {\rm ad}_{K_0}^j (X) \in\cA_{m-\mu+1+\mu-j}\subset \cA_{m-1}$. Then we get
\begin{align*}
\im [H_0, X ]
& = \im f'(K_0) [K_0, X]  + A_{m-1}\\
  &\stackrel{\eqref{sol2}}{=} \left(1-\eta(K_0)\right) \im [K_0, Y] + A_{m-1} \\
&\stackrel{\eqref{hom2}}{=}  \left(1-\eta(K_0)\right) \, \left( A - \la A \ra \right) + A_{m-1} \ , 
\end{align*}
with $A_{m-1}\in\cA_{m-1}$. 
Now put $R := -\eta(K_0) \, \left( A - \la A \ra \right) $. Since $x
\mapsto \eta(x) \in S^{-\infty}$, $R$ is a smoothing operator and thus
$A_{m-1}+R\in\cA_{m-1}$.
\end{proof}

\paragraph{The second homological equation.}
{We want to solve  eq. \eqref{hom.5}. Using the quasiperiodicity assumption $V(t) = W(\omega t)$, we look for a quasiperiodic solution $X_1(t) = X(\omega t)$ of the equation }
  \beq\label{hom4}
\omega\cdot\partial_\theta X(\omega t) + \im [H_0, X(\omega t)] =
W(\omega t) - \lan W(\omega t)\ran \ .  \eeq In order to define
precisely $ \lan W(\omega t)\ran $, consider again the vectors $\bv_j$
and the frequencies $\tilde \nu_j$ of Lemma \ref{lem.gio}. First   remark
that, since $\nu=\sum_{j=1}^{\tilde d}\tilde \nu_j \bv_j$, 
one has $\nu\cdot K=\sum_{j=1}^{\tilde d}(K\cdot \bv_j)\tilde \nu_j$,
so that, defining 
\begin{equation}
\label{ktilde}
\tilde K_j:=K\cdot \bv_j\ ,\qquad 
\tilde K:=(\tilde K_1,...,\tilde K_{\tilde
  d})\ ,
\end{equation}
one has 
$$
H_0 \equiv \nu\cdot K=\tilde \nu\cdot\tilde K \ ,
$$
and  furthermore, since
$\bv_j$ has integer entries, 
{then the joint spectrum of $\tilde
K\equiv(\tilde K_1,...,\tilde K_{\tilde d})$ is s.t. spec$(\tilde K)\subset
\Z^{\tilde d}+\tilde \lambda$, therefore the map $\R^{\tilde
  d}\ni\tau\mapsto A(\tau):=e^{\im\tau\cdot \tilde K}Ae^{-\im\tau\cdot \tilde K}
$ is periodic in each of the $\tau_j$'s.}
Define now 
\begin{equation}
\label{average2}
\la A\ra := \frac{1}{(2\pi)^{\tilde d}}\int_{\T^{\tilde d}}{\rm e}^{\im \tau
  \cdot \tilde K} \, A \, {\rm e}^{-\im \tau\cdot \tilde K} \di \tau  \ .
\end{equation}

\begin{remark}
\label{normal.form2}
Let $A \in \cA_m$, $m \in \R$. Then by Assumption $II^\prime$,  $\la A \ra \in \cA_m$ and
\begin{equation}
\label{eq.nf2}
[\tilde K_j, \la A \ra ] = 0 \ , \qquad  1 \leq j \leq \tilde d
\ ; \qquad [K_0,\la A \ra ]=0\ .
\end{equation}
\end{remark}

\begin{lemma}
\label{lem:hom2}
Let $A \in C^\infty_b\left(\T^n, \cA_m\right)$, $m \in \R$. Provided
\eqref{non.res3} holds, the homological equation \eqref{hom4} has a
solution $X \in C^\infty(\T^n, \cA_m)$. Furthermore if $A$ is
symmetric then $X$ is symmetric as well.
\end{lemma}
\begin{proof} 
For $A\in C^\infty(\T^{n},\cA_{m})$,  denote 
$A^\sharp(\theta, \tau) := {\rm e}^{\im\tau\cdot \tilde
  K}A(\theta){\rm e}^{-\im\tau\cdot \tilde K}$. By Assumption II$^\prime$,  $A^\sharp\in
C^\infty(\T^{n+\tilde d},\cA_{m})$. Since $A^\sharp$ is defined on $\T^{n+\tilde d}$, we can expand it  in  Fourier  series:
  $$
  A^\sharp(\theta,\tau) = \sum_{(k,\ell)\in\Z^{n+\tilde d}}\hat A^\sharp_{k,\ell} \, {\rm e}^{\im(k\cdot\theta +\ell\cdot \tau)}\,,
  $$
  where 
  $$
  \hat A^\sharp_{k,\ell} := \frac{1}{(2\pi)^{n+\tilde d}} \int_{\T^{n+\tilde d}}A^\sharp(\theta,\tau) {\rm e}^{-\im(k\cdot\theta +\ell\cdot \tau)} \, \di \theta \di \tau.
  $$
Notice that 
\beq
\label{X}
A(\theta) \equiv A^\sharp(\theta, 0) = \sum_{(k,\ell)\in\Z^{n+\tilde d}}\hat A^\sharp_{k,\ell}{\rm e}^{\im k\cdot\theta}.
\eeq\\
Then, instead of solving directly the  homological  equation \eqref{hom4}, we solve 
\beq
\label{hom3.1}
\omega\cdot\partial_\theta X^\sharp(\theta, \tau) + \im [H_0, X^\sharp(\theta, \tau)] = \left(W - \lan W\ran \right)^\sharp(\theta, \tau) \ , \qquad \forall \theta \in \T^n \ , \ \ \forall \tau \in \T^{\tilde d} \ .
\eeq
Clearly if we find a smooth solution $X^\sharp (\theta, \tau)$ of this equation, then $X(\theta) := X^\sharp(\theta, 0)$ solves the original homological equation \eqref{hom4}.
Now remark that
\begin{align*}
\im [H_0, X^\sharp(\theta, \tau)] & = 
\sum_{j=1}^{\tilde d} \tilde \nu_j \im [\tilde K_j, X^\sharp(\theta, \tau)] = 
\sum_{j=1}^{\tilde d} \tilde \nu_j \left. \frac{d}{d\epsilon} \right|_{\epsilon = 0}  {\rm e}^{\im\epsilon \tilde K_j} \, X^\sharp(\theta, \tau){\rm e}^{-\im\epsilon  \tilde K_j} \\
& = \sum_{j=1}^{\tilde d} \tilde \nu_j \left. \frac{d}{d\epsilon} \right|_{\epsilon = 0}  X^\sharp(\theta, \tau+ \epsilon {\bf e}_j) \\
&= \sum_{(k,\ell)\in \Z^{n+\tilde d}}  \hat X^\sharp_{k,\ell} \left.\frac{d}{d\epsilon} \right|_{\epsilon = 0} \sum_{j=1}^{\tilde d} \tilde \nu_j  \, {\rm e}^{\im(k\cdot\theta +\ell\cdot (\tau + \epsilon {\bf e}_j))}\\
& = \sum_{(k,\ell)\in \Z^{n+\tilde d}} \im \tilde \nu \cdot \ell \,  \hat X^\sharp_{k,\ell}  {\rm e}^{\im(k\cdot\theta +\ell\cdot \tau )} \ .
\end{align*}
Therefore, expanding in Fourier series, equation \eqref{hom3.1} is equivalent to 
$$
\im (\omega\cdot k + \tilde \nu\cdot \ell)\hat X_{k,\ell}^\sharp = \widehat{(W-\lan W\ran)}_{k,\ell}^\sharp \ .
$$
Hence define
\beq\label{hom5}
   \hat X_{k,\ell}^\sharp = -\im \frac{\widehat{(W-\lan W\ran)}_{k,\ell}^\sharp}{(\omega\cdot k+ \tilde \nu\cdot \ell)},\qquad {\rm if}\;\; \omega\cdot k + \tilde \nu\cdot \ell \neq 0.
\eeq
Since $W^\sharp$  is in $C^\infty(\T^{n+\tilde d}, \cA_{m})$ we get that for  any $j,N\geq 1$ there exists $ C_{N,j}$ such that
{$$
\wp_j^m\left( \widehat{(W-\lan W\ran)}_{k,\ell}^\sharp\right) \leq C_{N,j}(\vert k\vert +\vert\ell\vert)^{-N}.
$$
}So we get  easily that  if  $X$  is  defined by $X(\theta) = X^\sharp(\theta, 0)$ and $X^\sharp$ has Fourier coefficients \eqref{hom5}  with  $X^\sharp_{k,0}=0$,   then $X\in C_b^\infty(\T^n,\cA_{m})$.
\end{proof}

\subsection{The iterative Lemma}
\label{sec:iterative}
We state and prove the iterative Lemma which is the main
step for the proof of our main results. 

\begin{theorem}
\label{thm:smooth1}  Assume that the assumptions of Theorem
\ref{thm:main} or of Theorem \ref{thm:hosc} are satisfied.\\ There
exist $\delta>0$ and a sequence $\{X_j(t)\}_{j\geq 1}$ of self-adjoint
(time-dependent) operators in $\cH$ with 
$X_j\in C_b^\infty(\R,\cA_{\r-(\mu-1)-(j-1)\delta})$, such that $\forall
j$, the inequalities \eqref{X(t).est0} are
satisfied; {for any $N \geq 1$} the change of variables
\begin{equation}
\label{phij}
\psi = e^{-\im X_1(t)} \ldots e^{-\im X_N(t)} \vf  
\end{equation}
transforms $H_0 + V(t)$  into the Hamiltonian 
\begin{equation}
\label{hN}
H^\n(t) := H_0 + Z^\n(t) + V^\n(t)
\end{equation}
where $Z^\n\in
C_b^\infty(\R, \cA_\r)$  commutes with $K_0$, i.e. $[ Z^\n, K_0 ] = 0$, while   $V^\n\in C_b^\infty(\R, \cA_{\r-N\delta})$.
Furthermore, under the assumptions of Theorem \ref{thm:hosc}, one has
\begin{equation}
\label{kz}
[Z^{(N)};\tilde K_j]=0\ ,\quad \forall j=1,...,\tilde d\ .
\end{equation}
\end{theorem}

\begin{proof}[Proof]
It is proved by recurrence. Consider first the assumptions of Theorem
\ref{thm:main}. Using Lemmas \ref{T.1}, \ref{comX}, \ref{MR},
\ref{hom}, \ref{lem:hom2} one gets the theorem for $N=1$ with
$Z^{(1)}(t) := \lan V(t)\ran \in C^\infty_b(\R, \cA_\r)$. By Lemma
\ref{normal.form}, $[Z^{(1)}(t), K_0] = 0$. In this case $\delta$ can be taken 
as in \eqref{delta}. \\ The iterative step $N\to N+1$ is proved
following the same lines, just adding the remark that {$e^{\im
  X_{N+1}}Z^{(N)}e^{-\im X_{N+1}}-Z^{(N)}\in
\cA_{\rho-(\mu-1)-N\delta+\rho-1}$ $\subset \cA_{\rho-(N+1)\delta}$.}

Under the assumptions of Theorem \ref{thm:hosc}, the result is proved
along the same lines, with $\delta$ as in 
\eqref{delta*}. The property \eqref{kz} follows by Remark
\ref{normal.form2}.
\end{proof}

\subsection{Proof  of Theorem \ref{thm:main}}
By Theorem \ref{thm:smooth1}, the operator $H(t)$ is conjugated to
$H^\n(t)$. 
So we apply Theorem 1.5 of \cite{MaRo}
  to the 
Schr\"odinger equation for $H^{(N)}(t)$.
 More precisely we have
$$
[H^\n(t), K_0] = [V^\n(t), K_0] \in C^0_b(\R, \cA_{\r-N\delta}) 
$$
and thus,  by choosing $N$ large enough, \eqref{maro.est} ensures the result for the propagator $\cU_N(t,s)$ of $H^\n(t)$.

Now since $H(t)$ is conjugated to $H^\n(t)$, $H(t)$
generates a propagator $\cU(t,s)$ in the Hilbert space scale
$\cH^r$ unitarily equivalent to the propagator $\cU_N(t,s)$. 
Therefore, using also \eqref{X(t).est0}, 
the propagator  $\cU(t,s)$ fulfills  \eqref{pr1}, thus yielding  the result.\qed

\section{Applications}
In this section we prove Corollary \ref{Zoll}, Corollary \ref{Anha} and Corollary \ref{cor:rs.zoll}.

\subsection{Zoll manifolds}
 \label{zoll.app}
    To begin with we show how to put ourselves in the abstract setup. So first we  define the operator $K_0$. This will be achieved by exploiting the spectral properties of the operator $\triangle_g$.     
     Applying  Theorem 1  of Colin de Verdi\`ere \cite{cdv},  there  exists a pseudodifferential operator $Q$  of order $-1$, commuting with $\triangle_g$,
       such that ${\rm Spec}[\sqrt{\triangle_g} + Q]\subseteq \N+\lambda$  with some $\lambda\geq 0$. We can assume $\lambda>0$. {If not, denoting $\Pi_-$ the projector on the non positive eigenvalues, we replace $Q$ by $Q+C\Pi_-$ with $C>0$ large enough; remark that $\Pi_-$ commutes with $\triangle_g$ and is a smoothing operator.}
        So we define 
\begin{equation}
\label{zoll.k0}
K_0 := \sqrt{\triangle_g} + Q \ , \qquad H_0:=K_0^2  \ .
\end{equation}   
Now remark that 
         $H_0 = \triangle_g + 2Q\sqrt{\triangle_g} +Q^2$, 
          so  we have
        $$
        H_0=\triangle_g  + Q_0
        $$  
        where $Q_0$ is a pseudo-differential operator  of order $0$ and therefore
         $$
         H(t)= \triangle_g + V(t)  \equiv H_0 + \tilde V(t) \ ,
         \qquad \tilde V(t):= V(t) -Q_0
         $$
and we are in the setup of the abstract Schr\"odinger  equation \eqref{LS} with the new perturbation $\tilde V(t)$. 

Remark that $\cH^r := {\rm Dom } ((K_0)^r)$, $r \geq 0$,  coincides with the classical Sobolev space $H^r(M)$ and one has the equivalence of norms
$$
c_r \, \norm{\psi}_{H^r(M)} \leq \norm{\psi}_r \leq C_r \, \norm{\psi}_{H^r(M)} \ , \qquad \forall r \in \R \ . 
$$
We define the class $\cA_m$ to be the class of pseudodifferential operators whose (real valued) symbols belong to $S_{\rm cl}^m(M)$. 
{Clearly  $K_0 \in \cA_1$ (recall that $\Pi_-$ is a smoothing operator).}
It is classical that Assumptions I and II are fulfilled \cite{ho}.

 \begin{proof}[Proof of Corollary \ref{Zoll}]
 Assumption A holds true by construction of $K_0$, Assumption B holds  with $f(x) = x^2$ and therefore $\mu:=2$.
    Since $V(t)$ is a pseudodifferential operator of order $\r <2$ whose symbol belongs to $C^\infty_b(\R, S^\r_{\rm cl}(M))$,  one verifies easily, using pseudodifferential calculus (in particular estimates \eqref{est.1}--\eqref{est.3}), that  $\tilde V(t) = V(t) -Q_0 \in C^\infty_b(\R, \cA_\r)$.    Hence the   corollary  follows from Theorem \ref{thm:main}.
      \end{proof}

      \subsection{Anharmonic oscillators}
   We recall that for a symbol  $a$ (in the sense of Definition \ref{symbol.ao})  we denote by $a(x,D_x)$ its Weyl quantization
        \begin{equation}
        \label{weyl}
   \Big( a(x, D_x) \psi \Big)(x) := \frac{1}{2\pi} \iint_{y, \xi \in \R} {\rm e}^{\im (x-y)\xi} \, a\left( \frac{x+y}{2}, \xi \right) \, \psi(y) \, \di y \di \xi \ . 
        \end{equation}
    We endow $S^\r_\ao$ (defined in Definition \ref{symbol.ao}) with the family of seminorms
\begin{equation}
\label{seminorm}
\wp^\r_j(a) := \sum_{|\alpha| + |\beta| \leq j}
\ \ \sup_{(x, \xi) \in \R^{2}} \frac{\left|\partial_x^\alpha \, \partial_\xi^\beta a(x,\xi)\right|}{ \left[\tk_0(x,\xi)\right]^{\r-\frac{k\beta +l\alpha}{k+l}} } \ , \qquad j \in \N \ . 
\end{equation}
{The operator $K_0$ is defined using the spectral properties of the
Hamiltonian $H_{k,l}$ defined in \eqref{Hkl} that were studied in
detail in \cite{hero}; in that paper an accurate Bohr-Sommerfeld rule
for the the eigenvalues of $H_{k,l}$ was obtained and the
existence of a pseudodifferential operator $Q$ of order\footnote{
  Actually \cite{hero} proves that $Q$ has a symbol which is
  quasi-homogeneous of degree $-k-l$.  Here a symbol $f(x, \xi)$ is
  quasi-homogeneous of degree $m$ if
    $$f(\lambda^l x, \lambda^k \xi) = \lambda^m f(x,\xi) \ , \quad \forall \lambda >0 \ , \quad \forall (x, \xi) \in \R^2 \setminus \{ 0\} \ . $$
It is classical \cite{hero,hero2} that if $f$ is quasi-homogeneous of degree $m$, then it is a symbol in the class $S^{m/(k+l)}_\ao$.
} $-1${ such that ${\rm
  Spec}[H_{k,l}^{\frac{k+l}{2kl}}+Q]\subseteq \N+\lambda$
  ($\lambda\geq 0$) was proven.
}    Note that for our numerology $H_{k,l}^{\frac{k+l}{2kl}}$ is of
  order 1 by definition.}
Therefore we define
    $$
    K_0 := H_{k,l}^{\frac{k+l}{2kl}}+Q \ , \qquad H_0 := K_0^{\frac{2kl}{k+l}} \ .
    $$
   We define $\cA_m$ to be the class of pseudodifferential operator with symbols in $S^m_\ao$. Notice that by construction $\cA_m\subset \cL(\cH^s,\cH^{s-m})$ for all $s\in\R$.
    It is classical that $\cA$ fulfills Assumptions I and II (see \cite{hero, hero2}).\\ 
 On the other hand Assumptions A and B are fulfilled with $\mu := \frac{2kl}{k+l} > 1$ (as $k+l \geq 3$). Furthermore one has
     $$H_{k,\ell} = \left( K_0 - Q\right)^{\frac{2kl}{k+l}} = K_0^{\frac{2kl}{k+l}} + Q_0$$  where   $Q_0$  is a pseudodifferential operator of order  $\frac{2kl}{k+l }- 2$. 
Therefore     
    $$
    H(t) = H_{k,l} + V(t)  \equiv H_0 + \tilde V(t) \ , \qquad
    \tilde V(t) := V(t) + Q_0
    $$
    and once again we are in the setup of the abstract Schr\"odinger equation \eqref{LS} with the new perturbation $\tilde V(t)$.

   \begin{proof}[Proof of Corollary \ref{Anha}] 
        Since $V(t)$ is a pseudodifferential operator of order $\r<\frac{2kl}{k+l}$ whose symbol and its time-derivatives have uniformly (in time) bounded seminorms,   one verifies  that  $\tilde V(t) = V(t)+Q_0 \in C^\infty_b(\R, \cA_\r)$.    Hence the   corollary  follows from Theorem \ref{thm:main}.
    \end{proof}

\subsection{Relativistic Schr\"odinger equation on Zoll manifolds}
The proof of Corollary \ref{cor:rs.zoll} is along the lines developed
in Subsection \ref{zoll.app}. Let us remark that the operator
$\sqrt{\triangle_g +\mu}-\sqrt{\triangle_g}$ is of order $-1$. Hence,
defining $K_0$ as in \eqref{zoll.k0}, one has again $
\sqrt{\triangle_g +\mu} = K_0 + Q_0$ with $Q_0$ of order $-1$. Therefore
$$
H(t) = \sqrt{\triangle_g +\mu} + V(\omega t, x, D_x) = K_0 + \tilde V(\omega t)
$$
with the new perturbation $\tilde V(\omega t)\in C^\infty(\T^n, \cA_\r)$.\\

This time we verify Assumptions II$^\prime$, A$^\prime$ and B$^\prime$
with $d=1$ and $K_1 = K_0 = H_0$. Concerning the nonresonance
condition just remark that in this case we have that $\nu$ has only
one component given by 1.

Thus Theorem
\ref{thm:hosc} immediately yields Corollary \ref{cor:rs.zoll}.

\appendix 

\section{Technical lemmas on classical symbols}\label{sec:tec}

We begin with the following lemma whose proof is completely standard (and we skip it)

  \begin{lemma}
  \label{rem:class.symb}
 \begin{itemize}
\item[(i)] If $f \in S^a$, $g \in S^b$ then $f g \in S^{a+b}$.
\item[(ii)] If $f \in S^a$, then $f^{(j)} \in S^{a-j}$.
\item[(iii)]    If $x \mapsto \eta(x)$ is a smooth cut-off function on $\R$, then $\eta \in S^{- \infty}$.
\item[(iv)]  The function $f(x) = x^{a}$, $a >0$, is a classical elliptic symbol in $S^a_+$.
\end{itemize}
\end{lemma}
\begin{lemma}
\label{ref:fc}
If $g \in S^\mu$, $\mu \in \R$, then $g(K_0) \in \cA_\mu$.
\end{lemma}
\begin{proof}
By definition $g(x) = \sum_{0 \leq j \leq N-1} c_j x^{\mu - j} + R(x)$, $|R(x)| \leq C_N |x^{\mu - N}|$ for $|x| \geq 1$.
Then $g(K_0) = \sum_{0 \leq j \leq N-1} c_j K_0^{\mu - j} + R(K_0)$, where $R(K_0)$ is defined by functional calculus as
$R(K_0):= \int_0^\infty R(\lambda) \di E_{K_0}(\lambda) $, 
$\di E_{K_0}(\lambda)$ being the spectral resolution of $K_0$.
By Assumption I,  $\sum_{0 \leq j \leq N-1} c_j K_0^{\mu - j} \in \cA_\mu$ while the operator  $R(K_0)$ is $N$-smoothing (in the sense of Definition \ref{smoothing}). Since $N$ can be taken arbitrarily large, $g(K_0)$ fulfills Assumption I (v), therefore it belongs to $\cA_\mu$. The other properties are easily verified using such decomposition.
\end{proof}

Finally, we recall a commutator expansion lemma 
 following from   \cite[Lemma C.3.1]{dege}:
\begin{lemma}
\label{lem:gerard}
Let $f\in S^\r_+$  and $W\in \cA_m$. Then for all $N\geq [\r]$ we have
$$
[f(K_0), W] = \sum_{1\leq j\leq N}\frac{1}{j!}f^{(j)}(K_0){\rm ad}_{K_0}^jW + R_{N+1}(f,K_0,W),
$$
where 
$R_{N+1}(f,K_0,W) \in\cA_{[\r]+m-N}$.\\
Moreover if $W$ depends on time $t$ with uniform estimates in $\cA_m$ then it is also true for $R_{N+1}(f,K_0, W)$.
\end{lemma}
\begin{proof}
Apply \cite[Lemma C.3.1]{dege}  to  the bounded operator
$B= K_0^{-m}W$. 
\end{proof}

  \section{An abstract  proof of Egorov Theorem }\label{AET}
 In order to  check  Assumption II, we introduce the following condition 
 
 \vspace{0.5 cm}
 
    \noindent{\bf Assumption II-CL: } 
     For every $m\in\R$  and every  $A\in\cA_m$  there  exists  $\Phi^{(t)}(A)\in C^1_b(\R_t, \cA_m)$  and $R(A, t)\in C^0_b(\R_t, \cA_{m-1})$  such that 
      $\Phi^{(0)}(A)=A$  and 
     \beq\label{alh}
     \frac{d}{dt}\Phi^{(t)}(A) = \im^{-1}[\Phi^{(t)}(A),  K_0] + R(A, t)
     \eeq
    In applications in a pseudodifferential operator setting, we have $A= Op(a)$, $a$ is the symbol of $A$  and one can choose 
$\Phi^{(t)}(A) = Op(a\circ\phi^t)$ where $\phi^t$ is the classical flow of the symbol  of $K_0$.  Then  one has   to verify that 
$a\circ\phi^t$ belongs to the same symbol class as $a$ (see for example \cite{taylor}).
     \begin{theorem}[Abstract Egorov Theorem]\label{thm:AET}
     If  Assumption I and Assumption II-CL  are  satisfied then  Assumption II   holds true.
        \end{theorem}
\begin{proof}
   We follow \cite{robook} (p. 202-207).        Let $U(t) = {\rm e}^{-\im tK_0}$.   Compute
        \begin{align*}
        \frac{d}{d\tau}& \left(U(\tau-t)\Phi^{(\tau)}(A)U(t-\tau)\right) \\
     &   = U(\tau-t)\left(\im [\Phi^{(\tau)}(A),K_0] +\frac{d}{d\tau}\Phi^{(\tau)}(A)\right)U(t-\tau).
        \end{align*}
   So using \eqref{alh}   and integrate in $\tau$ between  0 and $t$ we get 
\beq\label{eg1}
   U(-t)AU(t) = \Phi^{(t)}(A)  + \int_0^t U(\tau-t)R(A,\tau)U(t-\tau) \di \tau.
\eeq
   Now we iterate from  this formula. In the following step  we apply this formula  for every $\tau$ to $A_{new}  =  R(A,\tau)$. So we get 
   \begin{align*}
     U(-t)A & U(t)  = A_0(t) +A_1(t) \\
     & + \int_0^t\int_0^{t-\tau}U(\tau+\tau_1-t)R(R(A,\tau),\tau-\tau_1)U(t-\tau-\tau_1)\di \tau \di \tau_1.
    \end{align*}
   where  $A_0(t) =  \Phi^{(t)}(A) $, $A_1(t) = \int_0^t \Phi^{(t-\tau)}(R(A,\tau)) \di \tau \in \cA_{m-1}$ and  \\$R(R(A,\tau),\tau-\tau_1)\in \cA_{m-2}$.\\
   At the step $N$ we get easily by induction:
\begin{align*}
   &  U(-t)AU(t)  = A_0(t) +A_1(t) + \cdots + A_N(t) \nonumber \\
   &\quad  +  \int_0^t\int_0^{t-\tau_0}\cdots\int_0^{t-\tau_0-\cdots-\tau_N} \di \tau_0 \di \tau_1\cdots \di \tau_N \nonumber \\
& \quad U(\tau_0+\tau_1+\cdots+\tau_N-t)R^{(N)}(A, \tau_0, \tau_1,\cdots,\tau_N)U(t-\tau_0-\tau_1-\cdots-\tau_N),
\end{align*}
where  $A_j\in C^0_b(\R,\cA_{m-j})$  and $R^{(N)}(A,\tau_1,\cdots,\tau_N)\in C_b^0(\R^{N+1}, \cA_{m-N-1})$.\\
Now we remark that the remainder term is as smoothing as we want by taking $N$ large enough, so the algebra being stable by smoothing perturbations
we get Assumption II. 
\end{proof}

 \section{Proof of Lemma \ref{lem.gio}}\label{tilde}

We reproduce here the proof given in the lecture notes by Giorgilli
\cite{gio} (in particular the technical results are contained in
Appendix A). A general presentation containing also the results that
we use here can be found in \cite{siegel}.

We start by stating without proof a simple Lemma.

\begin{lemma}
  \label{change}
Let $\be_1,...,\be_d$ and $\be'_1,...,\be'_d$ be two basis of $\Z^d$;
then the matrix $M=(M_{ij})$ s.t. $\be_i'=\sum_j M_{ij}\be_j$ is unimodular
with integer entries.  
\end{lemma}
Then one has the following corollary.
\begin{corollary}
  \label{det}
A collection of vectors $\be_j\in \Z^d$, $j=1,...,d$, is a basis of $\Z^d$ if and
only if the determinant of the matrix having $\be_j$ as rows is 1.
\end{corollary}
The corollary immediately follows from Lemma \ref{change} and the remark that
such a property holds for the canonical basis of $\Z^d$.

 Define now the resonance modulus $\cM_\nu$ of $\nu$ by
$$
\cM_\nu:=\left\{k\in\Z^d\ :\ \nu\cdot k=0\right\}\ .
$$ This is a discrete subgroup of $\R^d$ which satisfies
\begin{equation}
  \label{span}
{\rm span} (\cM_\nu)\cap\Z^d=\cM_\nu\ .
\end{equation}
Let $0\leq r\leq d-1$ be the dimension of $\cM_\nu$. It is well known
that any discrete subgroup of $\R^d$ admits a basis. Let
$\be_1,...,\be_r$, be a basis of $\cM_\nu$, and remark that the
vectors $\be_j$ have integer components.  Then the following result
holds\footnote{this can be found as Theorem 31 in \cite{siegel}, or as
  Lemma A.6 in \cite{gio}}.

\begin{lemma}
  \label{comple}
There exist $\tilde d:=d-r$ vectors $\bu_1,...,\bu_{\tilde d}$ with
integer entries, such that $\be_1,...,\be_r, \bu_1,...,\bu_{\tilde d}$
form a basis of $\Z^d$.
\end{lemma}

Then one obtains immediately the following
\begin{corollary}
  \label{M}
Let $M$ be the matrix with rows given by the vectors $\be_j$ and the
vectors $\bu_j$; define $\check\nu:=M\nu $, then one has
$\check\nu_i=0$, $\forall i=1,...,r$, while $\tilde
\nu_i:=\check\nu_{r+i}$, $i=1,...,\tilde d$ are independent over the
rationals.
\end{corollary}

\noindent{\it Proof of Lemma \ref{lem.gio}.} Consider the matrix
$M^{-1}$: since $M$ is unimodular with integer entries, the same is
true for $M^{-1}$, and one has $\nu=M^{-1}\check\nu$; however, since
the first $r$ components of $\check\nu$ vanish, such an expression
reduces to a linear combination of vectors with integer entries, the
coefficients of the combination being 
$\tilde\nu_1,...,\tilde \nu_{\tilde d}$. \qed

\end{document}